\documentclass[11pt]{amsart}
\usepackage[margin=1.35in]{geometry}
\usepackage{amscd,amsmath,amsxtra,amsthm,amssymb,stmaryrd,xr,mathrsfs,mathtools,enumerate,commath, comment}
\usepackage{stmaryrd}
\usepackage{xcolor}
\usepackage{commath}
\usepackage{comment}
\usepackage{tikz-cd}
\usepackage{longtable} 
\usepackage{pdflscape} 
\usepackage{booktabs}
\usepackage{hyperref}
\definecolor{darkgoldenrod}{rgb}{0.72, 0.53, 0.04}
\definecolor{vegasgold}{rgb}{0.77, 0.7, 0.35}
\definecolor{gold(metallic)}{rgb}{0.83, 0.69, 0.22}
\definecolor{sepia}{rgb}{0.44, 0.26, 0.08}
\definecolor{silver}{rgb}{0.75, 0.75, 0.75}
\usepackage{hyperref}
\hypersetup{
 colorlinks=true,
 linkcolor=red,
 filecolor=blue,
 citecolor=blue,
 urlcolor=red,
 pdftitle={Statistics for p-ranks of Artin--Schreier curves},
    }
\usepackage{enumerate}
\usepackage[utf8]{inputenc}
\usepackage[OT2,T1]{fontenc}
\usepackage{color}

\DeclareSymbolFont{cyrletters}{OT2}{wncyr}{m}{n}
\DeclareMathSymbol{\Sha}{\mathalpha}{cyrletters}{"58}

\numberwithin{equation}{section}
 
\usepackage{color}
\DeclareSymbolFont{cyrletters}{OT2}{wncyr}{m}{n}
\DeclareMathSymbol{\Sha}{\mathalpha}{cyrletters}{"58}

\newcommand{\Z}{\mathbb{Z}}

\newcommand{\op}[1]{\operatorname{#1}}
\newcommand{\F}{\mathbb{F}}

\theoremstyle{plain}
 \theoremstyle{definition}
\newtheorem{Th}{Theorem}[section]
\newtheorem{Lemma}[Th]{Lemma}

\newtheorem{Corollary}[Th]{Corollary}
\newtheorem{Proposition}[Th]{Proposition}
\newtheorem{Remark}[Th]{Remark}
 \theoremstyle{definition}
\newtheorem{Definition}[Th]{Definition}

\begin{document}

\title[Statistics for $p$-ranks of Artin--Schreier Covers]{Statistics for $p$-ranks of Artin--Schreier Covers}

\author[A.~Ray]{Anwesh Ray}
\address[Ray]{Department of Mathematics\\
University of British Columbia\\
Vancouver BC, Canada V6T 1Z2}
\email{anweshray@math.ubc.ca}

\subjclass[2010]{11G20, 11T06, 11T55, 14G17, 14H25}
\keywords{Arithmetic statistics, curves over finite fields, Artin--Schreier covers.}

\begin{abstract}
Given a prime $p$ and $q$ a power of $p$, we study the statistics of $p$-ranks of Artin--Schreier covers of given genus defined over $\F_q$, in the large $q$-limit. We refer to this problem as the \emph{geometric problem}. We also study an \emph{arithmetic} variation of this problem, and consider Artin--Schreier covers defined over $\F_p$, letting $p$ go to infinity. Distribution of $p$-ranks has been previously studied for Artin--Schreier covers over a fixed finite field as the genus is allowed to go to infinity. The method requires that we count isomorphism classes of covers that are unramified at $\infty$.
\end{abstract}

\maketitle
\section{Introduction}
\par The study of curves over finite fields leads to many interesting problems in arithmetic statistics. Recently, statistical questions have been framed and studied for curves varying in certain naturally occurring ensembles, see for instance, \cite{rudnick2008traces, kurlberg2009fluctuations, xiong2010fluctuations,bucur2011biased,entin2012distribution, thorne2014distribution, cheong2015distribution, bucur2016distribution, bucur2016statistics,sankar}. Throughout, $p$ will denote a prime number and $q$ will be a power of $p$. The field with $q$ elements is denoted $\F_q$. Given a curve $\mathcal{C}$ of positive genus $g$ over $\F_q$, its arithmetic is better understood by studying the structure of its Jacobian $\op{Jac}(\mathcal{C})$, which is a $g$-dimensional abelian variety. An important invariant associated to the curve is the $p$-rank of $\op{Jac}(\mathcal{C})(\bar{\F}_p)$, which we denote by $\tau$. This is the number such that $\op{Jac}(\mathcal{C})(\bar{\F}_p)\otimes_{\Z} (\Z/p\Z)$ is isomorphic to $\left(\Z/p\Z\right)^\tau$. The $p$-rank lies in the range $0\leq \tau \leq g$ and when $\tau=g$, the curve is \emph{ordinary}. In \cite{bucur2016statistics,sankar}, statistical questions are studied for $p$-ranks of Artin--Schreier covers of $\mathbb{P}^1$ defined over a fixed finite field $\F_q$, as the genus goes to infinity. It is shown that when $p$ is odd, the proportion of Artin--Schreier covers that are ordinary is zero as $g\rightarrow \infty$. When $p=2$, this proportion is shown to be non-zero, see \cite[Theorem 1.4]{sankar}.
\par 
We study similar questions about the distributions of $p$-ranks, however, the curves will no longer be defined over a fixed finite field. There are two independent questions we study in this paper. For the former, we fix the prime $p$ and let $q\rightarrow \infty$. This is referred to as the \emph{geometric problem}.
\par We make certain simplifications to the problem in question: 
\begin{itemize}
    \item the curves are not classified up to isomorphism, however, up to \emph{isomorphism of covers}, see Definition \ref{iso of covers}.
    \item We restrict to isomorphism classes of covers that are unramified at $\infty$. This is not a significant assumption since it can be guaranteed unless the cover is branched at all $\F_q$-points.
\end{itemize} 
\par In studying the arithmetic variant, the prime $p$ itself is allowed to vary. By a standard application of the Riemann-Hurwitz formula, the genus $g$ of an Artin--Schreier cover is always divisible by $\left(\frac{p-1}{2}\right)$. Since we are interested in the case when the genus is positive it follows that $g\rightarrow \infty$ as $p\rightarrow \infty$. In this setting, we fix the \emph{speed} at which the genus grows. Hence, we fix an integer $d>0$ throughout, and let $g(d,p):=d\left(\frac{p-1}{2}\right)$. It follows from the Deuring-Shafarevich formula \cite[Theorem 4.2]{subrao1975p} that the $p$-rank is of the form $\tau=r(p-1)$, where $r\geq 0$ is an integer. We fix $r$, thus the $p$-rank will also grow at a linear rate.

\par Thus, the two main questions studied are as follows.
\begin{enumerate}
    \item \emph{Geometric problem:} Given a prime $p$, what are the statistics for $p$-ranks of Artin--Schreier covers with fixed genus over $\F_q$, as $q\rightarrow \infty$?
    \item \emph{Arithmetic problem:} Suppose we are given an integer $d>0$. Then, what are the statistics for $p$-ranks of Artin--Schreier covers over $\F_p$ with genus $g(d,p)$, as $p\rightarrow \infty$.
\end{enumerate}
In section \ref{s3}, we prove the main results of the paper. Theorem \ref{main} gives a solution to the geometric problem. The results in the geometric context are compatible with results of Pries and Zhu \cite[Theorem 1.1]{pries2012p} and Maugeais \cite[Corollary 3.16]{maugeais2006quelques} which analyze the stratification of the moduli stack of Artin--Schreier curves over an algebraically closed field of characteristic $p$. The $p$-rank induces a stratification on the space of Artin--Schreier curves of genus $g$. For the stratum of genus $g=d\left(\frac{p-1}{2}\right)$ and $p$-rank $\tau=r(p-1)$, the irreducible components are given in terms of combinatorial data and are in bijection with partitions of $(d+2)$ into $(r+1)$ numbers satisfying further constraints. For further details, see Remark \ref{s 3 remark}.
\par Theorem \ref{main} provides an answer to the above problem and is also expressed in terms of partition data. This result applies to all odd primes $p$, and when $p=2$, it is required that $d$ is even. The result can be interpreted in terms of the distribution of points on the irreducible components of maximal dimension in the moduli of Artin--Schreier curves of fixed genus. We refer to Remark \ref{remark before main theorem} for further details.
\par The distribution of arithmetic data over the set of prime numbers is theme of central interest in number theory. Often the arithmetic data is associated with a \emph{global object} such as a variety over a number field or Galois representation. Famous examples of problems of this flavor include the \emph{Sato--Tate conjecture} for abelian varieties and \emph{Lehmer's conjecture} for elliptic curves. Although the \emph{arithmetic problem} above is not intrinsically associated with a global object such as a motive defined over a number field, it has certain similarities since it concerns arithmetic objects and the limit is taken as $p$ goes to infinity. Our results fit into a broader theme since the $p$-rank is the number of roots of the \emph{L-polynomial} that are $p$-adic units. There is significant interest in the study of the arithmetic of such polynomials and the properties of their associated Newton polygons, see for instance \cite{booher2020realizing,kramer2021p}. Theorem \ref{last main theorem} provides a solution to the above mentioned arithmetic problem.
\par We arrive at our results via combinatorial methods that are independent of the results and techniques in the above mentioned works. The contents of the paper are thus comprehensible to a wide audience. In section \ref{s 4}, the results of section \ref{s3} are illustrated through an explicit example.

\subsection*{Acknowledgements} The author participated in the workshop \href{https://sites.google.com/view/rethinkingnumbertheory/home}{RNT July 12-23, 2021}, in which he was introduced to the broader theme of arithmetic statistics for families of curves over finite fields. The author would like to thank the organizers Allecher Serrano L\'opez, Heidi Goodson and Mckenzie West for the marvelous experience. He would like to thank Soumya Sankar for helpful discussions. The author is very grateful to the anonymous referee for timely and thorough reading of the manuscript and for pointing out many substantial improvements that have been implemented in the final version.

\section{Preliminaries}
\label{section: preliminaries}
We fix a prime number $p$ and set $q=p^n$. Let $f(x)\in \F_q(x)$ be a non-constant rational function such that $f(x)\neq z^p-z$ for any $z\in \F_q(x)$. Associate to $f$ the Artin--Schreier cover $\mathcal{C}_f$ of $\mathbf{P}^1$, taken to be the projective closure of the affine curve defined by the equation $y^p-y=f(x)$. The map $\mathcal{C}_f\rightarrow \mathbf{P}^1$ sends $(x,y)\mapsto x$. Note that the Galois group of the cover is isomorphic to $\Z/p\Z$ and generated by $(x,y)\mapsto (x, y+1)$.
\begin{Definition}\label{iso of covers}
Two covers $\mathcal{C}_f$ and $\mathcal{C}_g$ associated to $f,g\in \F_q(x)$ are isomorphic if there is an isomorphism $\varphi:\mathcal{C}_f\xrightarrow{\sim} \mathcal{C}_g$ defined over $\F_q$, which fits into a commuting diagram:
\[
\begin{tikzcd}
\mathcal{C}_f \arrow{r}{\varphi} \arrow[swap]{d}{} & \mathcal{C}_g \arrow{d}{} \\
\mathbf{P}^1  \arrow{r}{\op{id}} & \mathbf{P}^1.
\end{tikzcd}
\]
We express this relation by simply writing $\mathcal{C}_f\simeq_{\mathbf{P}^1} \mathcal{C}_g$; note that it is stronger than requiring that $\mathcal{C}_f$ and $\mathcal{C}_g$ are isomorphic as curves.
\end{Definition}

Let $B\subset \mathbf{P}^1(\bar{\F}_p)$ be the set of poles of $f(x)$. By adjusting $f$ by a function of the form $z^p-z$, we may assume that $p$ does not divide the order of any pole of $f$. Then, the cover is ramified at precisely the points in $B$. We assume that the cover is unramified at $\infty$. Letting $d_\alpha$ be the order of the pole of $f(x)$ at $\alpha$, we set $\mathcal{D}_f=\sum_{\alpha\in B} d_\alpha \alpha$ to be the ramification divisor. We shall explicitly work with partial fractions. In setting up notation, for $\alpha\in B$, define
 \[x_\alpha:=\begin{cases}\frac{1}{(x-\alpha)}&\text{ if }\alpha\neq \infty,\\
 x&\text{ if }\alpha=\infty.
 \end{cases}\]
 Then, using a partial fraction decomposition, we write 
 \[f(x)=\sum_\alpha f_\alpha(x_\alpha),\] where $f_\alpha(x)$ is a polynomial of degree $d_\alpha$. The order of the pole of $f$ at $\alpha$ is $d_\alpha$. 

We wish to reduce the question of counting equivalence classes of Artin--Schreier covers to the simpler question of counting rational functions with prescribed properties. Following \cite{sankar}, we introduce a class of rational functions that are suitable for counting, which we shall refer to as \textit{admissible}. 
 
 \begin{Definition}\label{def2.2}A rational function $f(x)$ is \textit{admissible} if its partial fraction decomposition satisfies the following conditions 
 \begin{enumerate}
    \item\label{ad c1} At $\infty$, the cover $\mathcal{C}_f\rightarrow \mathbf{P}^1$ is unramified, in other words, the polynomial $f_\infty(x)$ is a constant.
     \item\label{ad c2} For each pole $\alpha$ of $f(x)$ such that $\alpha\neq \infty$, $f_\alpha$ has no constant term.
     \item\label{ad c3} If $p|j$, then the coefficient of $x^j$ in $f_\alpha$ is $0$.
 \end{enumerate}
 Let $\mathcal{S}$ be the set of admissible functions, and $\mathcal{S}(\F_q)$ the subset with coefficients in $\F_q$.
 \end{Definition}
 Note that \eqref{ad c1} implies that $f(x)$ can be expressed as a quotient $f(x)=g(x)/h(x)$, where $\op{deg} g(x)\leq \op{deg} h(x)$. Suppose that $f(x)$ is admissible. Let $d$ denote the following sum
\begin{equation}\label{def of d}d:=-2+\sum_\alpha (d_\alpha+1),\end{equation}where $\alpha$ runs over all poles of $f(x)$. In the formula above, $d_\alpha$ is set to be equal to $-1$ if there is no pole at $\alpha$. In particular, $d_\infty=-1$ since the cover is not ramified at $\infty$. By an application of the Riemann-Hurwitz formula, the genus of $\mathcal{C}_f$ is given by $g=d\left(\frac{p-1}{2}\right)$, see \cite[Proposition 3.7.8]{stichtenoth2009algebraic} for further details. Letting $r+1$ be the number of poles of $f(x)$, it follows from the Deuring-Shafarevich formula \cite[Theorem 4.2]{subrao1975p} that the $p$-rank is given by $\tau=r(p-1)$.
 
 Given rational functions $f$ and $g$, the following gives an explicit criterion for there to be an isomorphism $\mathcal{C}_f\simeq_{\mathbf{P}^1} \mathcal{C}_g$.
 \begin{Proposition}\label{boring prop}
Let $f$ and $g$ be in $\F_q(x)$. Then, there is an isomorphism of curves $\mathcal{C}_f\simeq \mathcal{C}_g$ over $\F_q$ if and only if $f(x)=ug(\gamma x)+\delta^p-\delta$ for $u\in \left(\Z/p\Z\right)^\times$, $\gamma \in \op{PSL}_2(\F_q)$ and $\delta\in \F_q(x)$. There is an isomorphism of covers $\mathcal{C}_f\simeq_{\mathbf{P}^1} \mathcal{C}_g$ if and only if $f(x)=ug(x)+\delta^p-\delta$ for $u\in \left(\Z/p\Z\right)^\times$ and $\delta\in \F_q(x)$. Furthermore, if $f$ and $g$ are both admissible and $\mathcal{C}_f\simeq_{\mathbf{P}^1} \mathcal{C}_g$, then, $g(x)=u f(x)$ for $u\in (\Z/p\Z)^\times$.
\end{Proposition}
 \begin{proof}
 The first part of the statement follows from the proof of \cite[Proposition 2.3]{sankar}. The remaining assertions are left as an exercise (see \cite[Remark 3.9]{pries2012p}).
 \end{proof}
 Henceforth, the word "isomorphism" shall refer to isomorphisms of covers in the sense of Definition \ref{iso of covers}.
 \begin{Corollary}\label{As cover adm}
 Any Artin--Schreier cover $\mathcal{C}\rightarrow \mathbf{P}^1$ that is unramified at $\infty$ is isomorphic to one that is of the form $\mathcal{C}_f\rightarrow \mathbf{P}^1$, with $f\in \mathcal{S}$. Moreover, if $\mathcal{C}\rightarrow \mathbf{P}^1$ is defined over $\F_q$, then, $f\in \mathcal{S}(\F_q)$.
 \end{Corollary}
 \begin{proof}
 Let $g$ be such that $\mathcal{C}=\mathcal{C}_g$. Let $\delta$ be such that for $f(x)=g( x)+\delta^p-\delta$ condition \eqref{ad c3} is satisfied. In greater detail, all monomials involving $x^j$ where $p|j$ in $f_\alpha(x)$ can be removed by adding $(\delta^p-\delta)$. It is easy to see that all other conditions are satisfied.
 \end{proof}
 \begin{Lemma}\label{d alpha}
 Let $f\in \mathcal{S}$ and $\alpha$ and $\alpha'$ that are conjugate over $\F_q$, then, $d_{\alpha}=d_{\alpha'}$.
 \end{Lemma}
 \begin{proof}
 Let $\sigma\in \op{Gal}(\bar{\F}_p/\F_q)$ be such that $\alpha'=\sigma(\alpha)$. Note that since $f(x)\in \F_q(x)$ we have that $(\sigma f)(x)=f(x)$. Here, $(\sigma f)(x)$ is the rational function obtained by applying $\sigma$ to the coefficients of $f(x)$.  This implies that 
 \[\sigma\left(f_{\alpha}(x_\alpha)\right)=f_{\alpha'}(x_{\alpha'}).\]On the other hand, \[\sigma\left(f_{\alpha}(x_\alpha)\right)=(\sigma f_{\alpha})\left(\sigma(x_{\alpha})\right)=(\sigma f_{\alpha})(x_{\alpha'}).\]
 Setting $y=x_{\alpha'}$, we find that $f_{\alpha'}(y)=\sigma(f_\alpha)(y)$, and hence, $d_{\alpha'}=d_\alpha$. 
 \end{proof}
 We introduce some further notation:
 \begin{itemize}
     \item Let $\mathcal{AS}_g(\F_q)$ be the set of all isomorphism classes of Artin--Schreier covers $\mathcal{C}\rightarrow \mathbf{P}^1$ defined over $\F_q$ that are of genus $g$ and unramified at $\infty$.
     \item Set $\mathcal{S}_g(\F_q)$ to be the set of all admissible rational functions $f(x)\in \F_q(x)$, such that $\mathcal{C}_f\in \mathcal{AS}_g(\F_q)$. The map $\mathcal{S}_g(\F_q)\rightarrow \mathcal{AS}_g(\F_q)$ is surjective by Corollary \ref{As cover adm}.
     \item Set $\mathcal{AS}_{g,\tau}(\F_q)$ (resp. $\mathcal{S}_{g,\tau}(\F_q)$) to be the subset of $\mathcal{AS}_{g}(\F_q)$ (resp. $\mathcal{S}_{g}(\F_q)$) for which the $p$-rank (resp. $p$-rank of the associated Artin--Schreier cover) is $\tau$.
 \end{itemize}
 
 \par We shall compute the large $q$-limit:
 \begin{equation}\label{kappa}\mathfrak{d}'(p, g, \tau):=\lim_{q\rightarrow \infty} \left(\frac{\# \mathcal{S}_{g,\tau}(\F_q)}{\# \mathcal{S}_{g}(\F_q)}\right),\end{equation}where $g>0$ and $q$ ranges over all powers of $p$. From this, it is easy to calculate
 \begin{equation}\label{delta}\mathfrak{d}(p, g, \tau):=\lim_{q\rightarrow \infty} \left(\frac{\# \mathcal{AS}_{g,\tau}(\F_q)}{\# \mathcal{AS}_{g}(\F_q)}\right).\end{equation}

\section{Main Results}\label{s3}
\par We prove the main results of this paper. First, we treat the geometric problem, after which we study the arithmetic variant, and see that the densities match up. In the next section, we provide an example to illustrate some of the results.
\subsection{Geometric problem: fixed prime, large $q$-limit}
\par Fix a prime number $p$ and let $g=d\left(\frac{p-1}{2}\right)$. Assume throughout that $d>0$. In this section we study the large $q$-limits $\mathfrak{d}'(p, g, \tau)$ and $\mathfrak{d}(p, g, \tau)$ defined in the previous section (cf. \eqref{kappa}, \eqref{delta}). We introduce some further notation. Let $f(x)\in \mathcal{S}$, note that this means that each polynomial $f_\alpha(x)$ satisfies the conditions of Definition \ref{def2.2}. For $\alpha\neq \infty$, we write $f_\alpha(x)=\sum_{i=1}^{d_\alpha} c_{i,\alpha} x^i$, where the terms $c_{i, \alpha}=0$ for all terms $i$ that are divisible by $p$. Thus, it follows from condition \eqref{ad c3} that $d_\alpha\not \equiv 0\mod{p}$. Note that it is assumed that $f_\infty(x)$ is a constant. For $j\in \Z_{\geq 1}$ with $j\not \equiv 0\mod{p}$, we set $g_j(x)$ to be the product
\[g_j(x):=\left(\prod_{d_\alpha=j} (x-\alpha)\right),\]
and set $g(x):=\prod_j g_j(x)^{j}$. Note that if $\alpha$ and $\alpha'$ are conjugate, then, by Lemma \ref{d alpha}, $d_\alpha=d_{\alpha'}$. It follows that $g_j(x)\in \F_q[x]$.
\par We write $f(x)=\frac{h(x)}{g(x)}$, where $\op{deg} h(x)\leq  \op{deg} g(x)$. Note that the polynomials $g_j(x)$ are squarefree polynomials. Set $\kappa_j:=\op{deg} g_j(x)$, note that 
\[d+2=\sum_j (j+1) \kappa_j,\] hence, this gives rise to a partition of $d+2$. Here, $\kappa_j$ is the multiplicity of $j+1$ as a term in the partition, and $j\not \equiv  0\mod{p}$. We have set $r+1$ to be the total number of poles of $f(x)$, and this is given by 
\[r+1=\sum_j \op{deg} g_j(x)=\sum_j \kappa_j.\] Recall from the previous section that the $p$-rank of $\mathcal{C}_f$ is given by $\tau=r(p-1)$. For $i=1,2,\dots$, set \begin{equation}\label{def of lambda}\lambda_i:=\#\{j|\kappa_j\geq i\},\end{equation} note that $r+1$ is the total number of terms in the partition of $(d+2)$ and $\lambda_i$ the number of distinct terms with multiplicity $\geq i$. In the case when $\kappa_j=0$ for all $j>1$, we have that $d$ is even and $\kappa_1=\frac{d+2}{2}$. In this case, $r=\frac{d}{2}$ and the $p$-rank is equal to the genus. Thus, the curve is ordinary precisely when the degree at each pole is $1$. There is only one partition that corresponds to this case when $d$ is even, and no partitions when $d$ is odd.
\begin{Definition}
Let $g=d\left(\frac{p-1}{2}\right)>0$ and $r\geq 0$. Denote by $\Omega_p(r+1,d+2)$ the set of partitions of $d+2$ into $r+1$ numbers that are $\not\equiv 1\mod{p}$, and set $N_p(r+1, d+2):=\# \Omega_p(r+1,d+2)$. 
\end{Definition}
Represent a partition in $\Omega_p(r+1, d+2)$ by the \emph{multiplicity vector}, i.e., the tuple of integers $\kappa=(\kappa_1, \dots, \kappa_t)$, where (by \eqref{def of d})
\[\sum_{j=1}^t (j+1) \kappa_j=d+2.\] Thus for a partition of $(d+2)$ into $(r+1)$ numbers, $\kappa_j$ is the multiplicity of $(j+1)$. Note that $\kappa_j=0$ for all $j\not\equiv 0\mod{p}$. The number $t$ above is chosen to be the largest number for which $\kappa_t\neq 0$. The number of terms in the associated partition is \[r+1=\sum_j \kappa_j.\] The multiplicity vector of $\mathcal{C}_f$ is denoted $\kappa(f)$. Given a partition in $\Omega_p(r+1, d+2)$, we find that 
\[2(r+1)=\sum_j 2\kappa_j\leq \sum_j (j+1)\kappa_j=d+2.\] Therefore we find that $r\leq \lfloor d/2\rfloor$.

\par \begin{Remark}\label{s 3 remark}In \cite{pries2012p}, Pries and Zhu consider $\mathcal{AS}_g$, the moduli space of genus $g=d\left(\frac{p-1}{2}\right)$ Artin--Schreier curves and study the stratification of $\mathcal{AS}_g$ according to $p$-rank. Given $\tau=r(p-1)$, let $\mathcal{AS}_{g, \tau}$ be the stratum consisting of Artin--Screier curves of $p$-rank $\tau$. Theorem 1.1. of \emph{loc. cit.} shows that the set of irreducible components of $\mathcal{AS}_{g, \tau}$ is in bijection with the set of partitions $\{e_1, \dots, e_{r+1}\}$ of $(d+2)$ into $(r+1)$ positive numbers such that each $e_j\not \equiv 1\mod{p}$. Thus, the set of irreducible components of $\mathcal{AS}_{g,\tau}$ is in bijection with $\Omega_p(r+1, d+2)$, and thus there are $N_p(r+1, d+2)$ irreducible components in total. The irreducible component of $\mathcal{AS}_g$ corresponding to the partition $\{e_1, \dots, e_{r+1}\}$ has dimension equal to 
\[d-1-\sum_{j=1}^{r+1}\left\lfloor \frac{(e_j-1)}{p} \right\rfloor.\] Thus in particular, the dimension is maximal when $e_j\leq p$ for all $e_j$. In terms of the multiplicity vector, this means that $\kappa_j=0$ for all $j\geq p$. Analogous results are proved for moduli of Artin--Schreier covers, i.e., Artin--Schreier curves along with a map to $\mathbb{P}^1$. In section 3.1 of \emph{loc. cit.}, Pries and Zhu introduce $\mathcal{AS}\op{cov}_g$, the moduli of Artin--Schreier covers of $\mathbb{P}^1$. We note here that the dimension of $\mathcal{AS}\op{cov}_g$ is $3$ more than $\mathcal{AS}_g$.\end{Remark}

\begin{Lemma}\label{basic lemma 3.3}
Let $N\geq 2$ be an integer. The number of monic squarefree polynomials of degree $N$ with coefficients in $\F_q$ is $\left(1-\frac{1}{q}\right) q^N$.
\end{Lemma}
\begin{proof}
The result is well known, see \cite[Theorem 2.2]{fulman2016generating}.
\end{proof}
Note that when $N=0$ or $1$, the above formula does not apply. In fact, when $N\leq 1$, the number of monic polynomials of degree $N$ with coefficients in $\F_q$ is $q^N$. Such polynomials are clearly squarefree.
\begin{Lemma}\label{ul bounds}
Let $\kappa=(\kappa_1, \dots, \kappa_t)$ be a partition of $d+2$ as above such that $r+1=\sum_j \kappa_j$. Set $N(\kappa)$ to be the number of rational functions $h(x)/g(x)\in \F_q(x)$ such that
\begin{enumerate}
    \item $h(x)$ and $g(x)$ are coprime,
    \item $\op{deg} h(x)\leq \op{deg} g(x)$, 
    \item $g(x)=\prod_j g_j(x)^j$, where $g_j(x)$ is a monic squarefree polynomial of degree $\kappa_j$,
    \item for $i\neq j$, $g_i(x)$ and $g_j(x)$ are coprime.
\end{enumerate}
Then, we have that
\[\left(1-\frac{d+2-r}{q}\right)^{\lambda_1+1} q^{d+3}\leq N(\kappa)\leq \left(1-\frac{1}{q}\right)^{\lambda_2} q^{d+3},\]where $\lambda_i$ is the number of $\kappa_j$ that are $\geq i$, cf. \eqref{def of lambda}.
\end{Lemma}
\begin{proof} Note that the degree of $g(x)$ is $(d+1-r)$. For $i=1,\dots, t$, we set 
\[\beta_i:=\begin{cases} 0 & \text{ if }\kappa_i\leq 0,\\
1 & \text{ if }\kappa_i>  0.
\end{cases} \text{ and }\xi_i:=\begin{cases} 0 & \text{ if }\kappa_i\leq 1,\\
1 & \text{ if }\kappa_i>  1
\end{cases}\]We observe that $\lambda_1=\sum_{i=1}^t \beta_i$ and $\lambda_2=\sum_{i=1}^t \xi_i$. By Lemma \ref{basic lemma 3.3}, the number of squarefree monic polynomials over $\F_q$ of degree $\kappa_1$ is $\left(1-\frac{1}{q}\right)^{\xi_1} q^{\kappa_1}$, and this is the number of choices for $g_1(x)$. For a given choice of $g_1(x)$, the number of choices for $g_2(x)$ that are coprime to $g_1(x)$ is $\leq \left(1-\frac{1}{q}\right)^{\xi_2} q^{\kappa_2}$. Assume that $\kappa_2>0$. Since there are at most $(d+1-r)$ factors that divide $g_1(x)$, the total number of choices for $g_2(x)$ is \[\geq \left(1-\frac{1}{q}\right) q^{\kappa_2}-(d+1-r) q^{\kappa_2-1}=\left(1-\frac{d+2-r}{q}\right)q^{\kappa_2}.\]
If $\kappa_2=0$, we find that the number of choices of $g_2(x)$ is $q^{\kappa_2}=1$. Combining the above statements, we find that the total number of choices for $g_2(x)$ is $\geq \left(1-\frac{d+2-r}{q}\right)^{\beta_2}q^{\kappa_2}$. Suppose we have made choices for $g_1(x),\dots, g_j(x)$, and would like to choose $g_{j+1}(x)$ to be coprime to $g_1(x),\dots, g_j(x)$. The same argument tells us that the number of choices for $g_{j+1}(x)$ lies between $ \left(1-\frac{d+2-r}{q}\right)^{\beta_{j+1}}q^{\kappa_{j+1}}$ and $\left(1-\frac{1}{q}\right)^{\xi_{j+1}} q^{\kappa_{j+1}}$. Therefore, we find that the number of choices for $g(x)$ lies in between \[ \left(1-\frac{d+2-r}{q}\right)^{\lambda_1} q^{r+1}\] and $\left(1-\frac{1}{q}\right)^{\lambda_2} q^{r+1}$. Note that $\op{deg}g(x)=d+1-r$ and $h(x)$ has degree $\leq \op{deg} g(x)$. We find that for a given choice of $g(x)$, the number of choices for $h(x)$ lies in between \[q^{d+2-r}-(d+1-r)q^{d+1-r}=\left(1-\frac{d+1-r}{q}\right)q^{d+2-r}\geq \left(1-\frac{d+2-r}{q}\right)q^{d+2-r}\] and $q^{d+2-r}$. Therefore, we obtain the bounds
\[\left(1-\frac{d+2-r}{q}\right)^{\lambda_1+1} q^{d+3}\leq  N(\kappa)\leq \left(1-\frac{1}{q}\right)^{\lambda_2} q^{d+3}.\]
\end{proof}
\begin{Definition}Let $\mathcal{S}_{g}^{\kappa}(\F_q)$ be the subset of $\mathcal{S}_g(\F_q)$ consisting of all $f(x)$ for which $\kappa(f)=\kappa$.
\end{Definition}
\par We review some standard notation. Given two functions $F(q)$ and $G(q)$, we recall that $F(q)\sim G(q)$ if $\lim_{q\rightarrow \infty} F(q)/G(q)=1$. In the limit above $q=p^n$, where $n\rightarrow \infty$. We write $F(q)=O\left(G(q)\right)$ if there is a constant $C>0$ such that $F(q)\leq C G(q)$ for \emph{all} values of $q$. We write $F(q)=o\left(G(q)\right)$ if $\lim_{q\rightarrow \infty} F(q)/G(q)=0$.
\par Next, we prove estimates for the size of $\mathcal{S}_g(\F_q)$. Note that if $\kappa_j=0$ for all $j\geq p$, then the associated Artin--Schreier curve lies on an irreducible component of $\mathcal{AS}_g$ of maximal dimension (cf. Remark \ref{s 3 remark}). The result below shows that $\mathcal{S}_g(\F_q)\sim q^{d+3}$ for partitions such that $\kappa_j=0$ for all $j\geq p$, and $\mathcal{S}_g(\F_q)=o\left( q^{d+3}\right)$ if $\kappa_j>0$ for some value of $j$ which is larger than $p$.
\begin{Lemma}\label{lemma 3.1}
Let $\kappa\in \Omega_p(r+1, d+2)$. Let $\lambda_1$ and $\lambda_2$ be the quantities defined in the previous section, cf. \eqref{def of lambda}. We have the following assertions:
\begin{enumerate}
    \item\label{lemma 3.1 p1} Suppose that $\kappa_j=0$ for all $j\geq p$. Then, we have that 
\[ \left(1-\frac{d+2}{q}\right)^{\lambda_1+1}  q^{d+3}\leq \# \mathcal{S}_{g}^{\kappa}(\F_q)\leq \left(1-\frac{1}{q}\right)^{\lambda_2} q^{d+3},\] in particular, as $q\rightarrow \infty$,
\[\# \mathcal{S}_{g}^{\kappa}(\F_q)\sim  q^{d+3}.\]
\item Suppose that $\kappa_j\neq 0$ for some $j\geq p$. Then, we have that
\[\# \mathcal{S}_{g}^{\kappa}(\F_q)= o(q^{d+3}).\]
\end{enumerate} 
\end{Lemma}
\begin{proof}
Recall that we have expressed $f(x)=h(x)/g(x)$, where $g(x)=\prod_j g_j(x)^j$ and $\kappa_j=\op{deg} g_j(x)$, where $g_j(x)$ is a squarefree polynomial with coefficients in $\F_q$. Assume without loss of generality that $g_j(x)$ is a monic polynomial for all $j$. Note that the degree of $g(x)$ is \[\sum_j j \kappa_j=d+2-(r+1)=d-r+1.\] Since $f_{\infty}(x)$ is a constant, we find that $\op{deg} h(x)\leq \op{deg} g(x)$. Furthermore, note that by construction, the polynomials $g_i(x)$ and $g_j(x)$ are coprime for $i\neq j$.

\par First, we consider the case when $\kappa_j=0$ for all $j\geq p$. In this case, for every pole $\alpha$ of $f(x)$, we have that $d_{\alpha}<p$. Thus, the condition \eqref{ad c3}, which requires that the coefficient of $x^j$ be zero in $f_\alpha(x)$ whenever $p|j$, is automatically satisfied. Furthermore, $h(x)$ is coprime to $g(x)$. Therefore, the result in this case follows from Lemma \ref{ul bounds}. In greater detail, the lower bound proceeds from the following observation \[\left(1-\frac{d+2}{q}\right)\leq \left(1-\frac{d+2-r}{q}\right).\]
\par Next, assume that for some $j>p$, we have that $\kappa_j\neq 0$. In this case, the same calculation goes through with the one difference being that there is an additional constraint requiring that the coefficient of $x^{j}$ in $f_\alpha(x)$ is $0$ for all $j>0$ such that $p|j$. This condition was automatically satisfied in the previous case. In this case, this gives an additional condition on $h(x)$, which forces the number of choices to be
\[\leq \left(1-\frac{1}{q}\right)^{\lambda_2} q^{d+2}=o(q^{d+3}).\] Let us explain this in greater detail. 

\par Since $\kappa_j\neq 0$ for some $j>p$, there exists a pole $\alpha$ of $f(x)$ such that $d_{\alpha}=j>p$. The coefficient of $x^p$ in $f_\alpha(x)$ is required to be $0$. Set $\eta(x)\in \F_q(x)$ to denote the function given by 
\[\eta(x):=\sum_{\alpha'} \frac{1}{(x-\alpha')},\]where $\alpha'$ ranges over the $\F_q$-conjugates of $\alpha$. Note that $\eta(x)^p=\sum_{\alpha'} \frac{1}{(x-\alpha')^p}$. This function $\eta(x)^p$ has coefficients in $\F_q$ and is expressed as $\eta(x)^p=\frac{\eta_1(x)}{\eta_2(x)}$, with \[\op{deg} \eta_1(x)\leq \op{deg} \eta_2(x)\leq \kappa_j.\] Moreover, note that $\eta_2(x)$ divides $\prod_{\alpha'} (x-\alpha')^p$, and hence divides $\prod_{\alpha'} (x-\alpha')^j$. Since this latter product divides $g(x)$, it follows that $\eta_2(x)$ divides $g(x)$. For any choice of $0\neq c\in \F_q$, the function $f_c(x):=f(x)+c \eta(x)^p$ does not satisfy condition \eqref{ad c3} since now the coefficient of $x^p$ in $f_{c,\alpha}(x)$ is nonzero, and $p<j=d_\alpha$. We may express $f_c(x)=h_c(x)/g(x)$, where 
\[\op{deg} h_c(x)\leq \op{deg} g(x).\] \par The number of functions $f(x)=h(x)/g(x)$ satisfying the conditions of Lemma \ref{ul bounds} is $\leq \left(1-\frac{1}{q}\right)^{\lambda_2} q^{d+3}$.
\par On the other hand, from an admissible function $f(x)=\frac{h(x)}{g(x)}$, we have constructed $q$ new functions $f_c(x)$, as $c$ ranges over $\F_q$. Only one of which is admissible (i.e., when $c=0$). It is easy to see that if $f(x)$ and $F(x)$ are distinct admissible functions, then $f_c(x)\neq F_{c'}(x)$ for all $c,c'$. Hence, the total number of admissible functions for the partition vector $\kappa$ is \[\leq \frac{1}{q}\left(\left(1-\frac{1}{q}\right)^{\lambda_2} q^{d+3}\right)=O(q^{d+2})=o(q^{d+3}).\]
\end{proof}
\begin{Definition}
Let $M_p(r+1, d+2)$ be the number of partitions $\kappa\in \Omega_p(r+1, d+2)$ such that $\kappa_j=0$ unless $1\leq j<p$. Set \[M_p(d+2):=\sum_{r=0}^{\lfloor d/2\rfloor} M_p(r+1,d+2).\]
\end{Definition}

\begin{Lemma}\label{Mp nonzero}
With respect to notation above, we have that $M_p(d+2)=0$ if and only if $p=2$ and $d$ is odd.
\end{Lemma}
\begin{proof}
The proof of the result is rather straightforward and not particularly interesting, hence it is omitted.
\end{proof}
\begin{Remark}\label{remark before main theorem}
We note here that by the discussion in Remark \ref{s 3 remark}, $M_p(d+2)$ (resp. $M_p(r+1, d+2)$) is equal to the number of irreducible components of $\mathcal{AS}_g$ (resp. $\mathcal{AS}_{g, \tau}$) of dimension $d-1$. These are the components of maximal dimension. The fraction $\frac{M_p(d+2)}{M_p(r+1,d+2)}$ plays a role in the next result. It follows from the aforementioned assertions that this fraction is the proportion of irreducible components of maximal dimension in $\mathcal{AS}_g$ parametrizing Artin--Schreier covers with $p$-rank $\tau$. Thus, the Theorem below has a suitable geometric interpretation.
\end{Remark}

\begin{Th}\label{main}
Fix a prime number $p$ and let $r,d\geq 0$ such that $r\leq \lfloor d/2\rfloor$. Assume that $d$ is even when $p=2$. Set $g=d\left(\frac{p-1}{2}\right)$ and $\tau=r(p-1)$, and consider the limits $\mathfrak{d}'(p,g,\tau)$ (cf. \eqref{kappa}) and $\mathfrak{d}(p,g,\tau)$ (cf. \eqref{delta}). With respect to notation above, we have that 
\[\mathfrak{d}'(p,g,\tau)=\mathfrak{d}(p,g,\tau)=\frac{M_p(r+1, d+2)}{M_p(d+2)}.\]
\end{Th}
\begin{proof}
Since it is assumed that $d$ is even for $p=2$, it follows from Lemma \ref{Mp nonzero} that $M_p(d+2)$ is not equal to $0$. Note that
\[\# \mathcal{S}_{g,\tau}(\F_q)=\sum_{\kappa \in \Omega_p(r+1,d+2)} \# \mathcal{S}_{g}^\kappa(\F_q),\] where we recall that $g=d\left(\frac{p-1}{2}\right)$ and $\tau= r(p-1)$. It follows from Lemma \ref{lemma 3.1} that $\# \mathcal{S}_{g,\tau}(\F_q)\sim M_p(r+1, d+2)q^{d+3}$. Likewise, we have that $\# \mathcal{S}_{g}(\F_q)\sim M_p( d+2)q^{d+3}$, and it thus follows that 
\[\mathfrak{d}'(p,g,\tau)=\frac{M_p(r+1, d+2)}{M_p(d+2)}.\] 
\par On the other hand, if $f$ and $g$ are both admissible, then, by Proposition \ref{boring prop}, \[\mathcal{C}_f\simeq_{\mathbf{P}^1} \mathcal{C}_g\Leftrightarrow f=ug\text{ for }u\in (\Z/p\Z)^\times.\] Hence, each isomorphism class of covers (unramified at $\infty$) consists of exactly $(p-1)$ admissible functions. As a result, 
\[\# \mathcal{AS}_{g,\tau}(\F_q)\sim \frac{M_p(r+1, d+2)}{p-1}q^{d+3} \text{ and }\# \mathcal{AS}_{g}(\F_q)\sim \frac{M_p( d+2)}{p-1}q^{d+3},\] and consequently, $\mathfrak{d}(p,g, \tau)=\mathfrak{d}'(p, g, \tau)$.
\end{proof}

Note that $\mathfrak{d}(p,g, \tau)$ depends on $p$.

\begin{Corollary}\label{cor 3.10}
Fix a prime number $p$ and let $r,d\geq 0$ such that $r\leq \lfloor d/2\rfloor$. Assume that $d$ is even when $p=2$. Suppose that $(r+1)p<d+2$, then, $\mathfrak{d}(p, g, \tau)=0$. In particular, if $p<d+2$, then, the proportion of isomorphism classes of Artin--Schreier covers of genus $g=d\left(\frac{p-1}{2}\right)$ and $p$-rank $\tau=0$ is zero in the large $q$-limit.
\end{Corollary}
\begin{proof}
Note that when $(r+1)p<d+2$, it is not possible to partition $d+2$ into $r+1$ numbers, all of which are $\leq p$. In this case, $M_p(r+1, d+2)=0$, and the result follows from Theorem \ref{main}.
\end{proof}

\begin{Remark}
It follows from Remark \ref{s 3 remark} that since $M_p(1,d+2)=0$, the irreducible components of the stratum of $\mathcal{AS}_g$ for which $\tau=0$ all have dimension $<(d-1)$. These components contain fewer points than other components of $\mathcal{AS}_g$ with dimension $(d-1)$.
\end{Remark}
\subsection{Arithmetic problem: large $p$-limit}
\par We now study the \emph{arithmetic} variant of the problem considered in the previous section. Many of the calculations are similar in spirit. Fix $d>0$, $r\geq 0$, and let $g=g(d,p)=d\left(\frac{p-1}{2}\right)$, $\tau=\tau(r,p):=r(p-1)$. Note that $g$ and $\tau$ depend on $p$ and are increasing at fixed speeds as $p\rightarrow \infty$. This is suppressed in the notation.

\par Set $\mathfrak{p}'(d, r)$ to be the limit 
\begin{equation}\label{def p prime}\mathfrak{p}'(d,r):=\lim_{p\rightarrow \infty} \frac{\#\mathcal{S}_{g}^{\tau}(\F_p)}{\#\mathcal{S}_{g}(\F_p)},\end{equation} and 
set 
\begin{equation}\label{def p frak}\mathfrak{p}(d,r):=\lim_{p\rightarrow \infty} \frac{\#\mathcal{AS}_{g}^{\tau}(\F_p)}{\#\mathcal{AS}_{g}(\F_p)},\end{equation}
with $g=g(d,p)$ and $\tau=\tau(r,p)$.
\begin{Definition}
Let $\Theta(r+1, d+2)$ be the set of all partitions of $d+2$ into $r+1$ numbers that are all $\geq 2$. Denote by $\Theta(d+2)$ the set of all partitions of $d+2$ into numbers $\geq 2$. We set $T(r+1, d+2):=\# \Theta(r+1, d+2)$ and $T(d+2):=\#\Theta(d+2)$.
\end{Definition}
It is clear that $T(d+2)\neq 0$. A partition in $\Theta(r+1, d+2)$ is represented by a partition vector $\kappa=(\kappa_1, \dots, \kappa_t)$, where $\kappa_j$ is the multiplicity of $(j+1)$ in the partition. There is no dependence on $p$ in the definition of $T(r+1,d+2)$. Note that 
\begin{equation}T(d+2)=\sum_{r=0}^{\lfloor d/2\rfloor} T(r+1,d+2).\end{equation}
\begin{Lemma}\label{lemma 3.2}
Let $\kappa\in \Theta(r+1, d+2)$ and suppose that $\kappa_j=0$ for all $j\geq p$. Let $\lambda_1$ and $\lambda_2$ be the quantities defined in \eqref{def of lambda}. Then, we have that 
\[\left(1-\frac{d+2}{p}\right)^{\lambda_1+1}  p^{d+3}\leq \# \mathcal{S}_{g}^{\kappa}(\F_p)\leq \left(1-\frac{1}{p}\right)^{\lambda_2} p^{d+3}.\] 
\end{Lemma}

\begin{proof}
The result follows from the first part of Lemma \ref{lemma 3.1}.
\end{proof}
\begin{Corollary}\label{last boring corollary}
Let $\kappa\in \Theta(r+1, d+2)$, and $g=g(d,p)$, $\tau=\tau(r,p)$. We have that \[\# \mathcal{S}_{g}^{\kappa}(\F_p)\sim p^{d+3},\] as $p\rightarrow \infty$.
\end{Corollary}
\begin{proof}
The condition requiring $\kappa_j=0$ for all $j\geq p$ is satisfied when $p$ is large, and the result thus follows from Lemma \ref{lemma 3.2}.
\end{proof}

In the result below, there is no constraint on $d$ when $p=2$. This is because $T(d+2)$ is always non-zero, even when $p=2$ and $d$ is odd. Recall that this was not the case for $M_p(d+2)$ in the statement of Theorem \ref{main}, which is why the case when $p=2$ and $d$ odd was excluded from the assertion.
\begin{Th}\label{last main theorem}
Fix a prime number $p$ and let $r,d\geq 0$ such that $r\leq \lfloor d/2\rfloor$. Recall that $\mathfrak{p}'(d,r)$ are $\mathfrak{p}(d,r)$ the limits defined by \eqref{def p prime} and \eqref{def p frak} respectively. Then, we have that
\[\mathfrak{p}(d, r)=\mathfrak{p}'(d, r)=\frac{T(r+1,d+2)}{T(d+2)}.\]
\end{Th}
\begin{proof}
The proof is similar to that of Theorem \ref{main}. Assume that $p$ is large enough so that $\kappa_j=0$ for all $j\geq p$. Recall that by abuse of notation, $g:=g(d,p)$, $\tau=\tau(r,p)$. Since $d$ and $r$ are fixed, $g$ and $\tau$ increase linearly with $p$ (but this is suppressed in our notation). We have that 
\[\# \mathcal{S}_{g,\tau}(\F_p)=\sum_{\kappa\in \Theta(r+1,d+2)} \# \mathcal{S}_g^\kappa(\F_p).\] 
It follows from Lemma \ref{lemma 3.2} and Corollary \ref{last boring corollary} that as $p\rightarrow \infty$,
\[\# \mathcal{S}_{g, \tau}(\F_p)\sim T(r+1, d+2) p^{d+3}\text{ and } \# \mathcal{S}_{g, \tau}(\F_p)\sim T(d+2) p^{d+3}.\] Putting everything together, we obtain that
\[\mathfrak{p}'(d,r)=\frac{T(r+1,d+2)}{T(d+2)}.\]

That $\mathfrak{p}(d,r)=\mathfrak{p}'(d,r)$ follows from the equalities:
\[\# \mathcal{AS}_{g, \tau}(\F_p)\sim \frac{1}{(p-1)}\# \mathcal{S}_{g, \tau}(\F_p)\text{ and }\# \mathcal{AS}_{g, }(\F_p)\sim \frac{1}{(p-1)}\# \mathcal{S}_{g, }(\F_p),\] see the proof of Theorem \ref{main}.
\end{proof}

\section{Some examples}\label{s 4}

\par We illustrate results proved in the previous section through examples.

\subsection{The geometric case}
\par First, let us pick the prime $p=5$, and let $d=8$. The genus is given by $g=d\left(\frac{p-1}{2}\right)=8\times 2=16$. The $p$-rank $\tau=r(p-1)=4r$. The values of $r$ range from $r=0$ to $r=4$. Note that when $r=4$, $\tau=16=g$, and this is the case when the curve is ordinary. 

\par We compute the proportion of isomorphism classes of Artin--Schreier covers with genus $g=16$ and $5$-rank $\tau=4r$, in the large $q$-limit. There are a total of $42$ partitions of $10$, however, there are constraints on the partitions we consider. Recall that $M_p(r+1,d+2)=M(r+1, 10)$ is the number of partitions
\[10=\sum_j (j+1)\kappa_j\] for which $\kappa_j=0$ for $j$ outside the range $1\leq j<5$. Thus, we are to use only numbers $2,3,4,5$ in partitioning $10$ into $r+1$ numbers. Let's list the partitions for each choice of $r=0,\dots, 4$.

\begin{enumerate}
    \item \emph{r=0:} There is no partition, since $10=10$ is not allowed. Hence, $M(1, 10)=0$. Thus, the proportion of curves with $5$-rank $0$ is $0$ in the $q$-limit.
    \item \emph{r=1:} We write $10=a+b$, where $a,b\in \{2,\dots, 5\}$. The only partition is 
    \[10=5+5,\] hence, $M(2, 10)=1$.
    \item \emph{r=2:} We write $10=a+b+c$, where $a,b,c\in \{2,\dots, 5\}$. The partitions are 
    \[10=5+3+2=4+4+2=4+3+3,\] hence, $M(3, 10)=3$.
    \item \emph{r=3:}
    We write $10=a+b+c+d$ this time. Since they are all $\geq 2$, there are only two choices:
    \[10=2+2+2+4=2+2+3+3,\]
    hence, $M(4, 10)=2$.
    \item \emph{r=4:} We come to single ordinary case. The only partition is 
    \[10=2+2+2+2+2,\] and $M(5,10)=1$.
\end{enumerate}
Putting it all together, we find that \[M(10)=\sum_{r=1}^4 M(r+1, 10)=0+1+3+2+1=7.\] Therefore, according to Theorem \ref{main}, the proportion of isomorphism classes of covers of genus $g=16$ and $p$-rank $\tau=4r$ is given by the following proportions:
\[\begin{split}
    &\mathfrak{d}(5,16,0)=\frac{0}{7}, \\
    &\mathfrak{d}(5,16,4)=\frac{1}{7},\\ &\mathfrak{d}(5,16,8)=\frac{3}{7}, \\
    &\mathfrak{d}(5,16,12)=\frac{2}{7}, \\
    &\mathfrak{d}(5,16,16)=\frac{1}{7}. 
\end{split}\]

\subsection{The arithmetic case}
\par Let's forget the prime $p$, and study the arithmetic problem for $d=8$. The value of $r$ ranges from $0$ to $4$ and specifies a given speed at which $\tau$ increases with $p$. In this setting, we no longer have the constraint that the partitions should involve terms from $\{2,\dots, 5\}$, but only that the terms be $\geq 2$. One again, we enumerate partitions for the five values of $r$.

\begin{enumerate}
    \item \emph{r=0:} This time, $10=10$ \emph{is} allowed. Hence, $T(1, 10)=1$.
    \item \emph{r=1:} We write $10=a+b$, where $a,b\geq 2$. There are $4$ partitions
    \[10=2+8=3+7=4+6=5+5,\] hence, $T(2, 10)=4$.
    \item \emph{r=2:} We write $10=a+b+c$. The partitions are 
    \[10=6+2+2=5+3+2=4+4+2=4+3+3,\] hence, $T(3, 10)=4$.
    \item \emph{r=3:}
    We write $10=a+b+c+d$ this time. Since they are all $\geq 2$, there are only two choices:
    \[10=2+2+2+4=2+2+3+3,\]
    hence, $T(4, 10)=2$.
    \item \emph{r=4:} We come to single ordinary case. The only partition is 
    \[10=2+2+2+2+2,\] and $T(5,10)=1$.
\end{enumerate}
Hence, we have that $T(10)=12$. According to Theorem \ref{last main theorem} the proportions are as follows:
\[\begin{split}
  &  \mathfrak{p}(d,0)=\frac{1}{12},\\
   &  \mathfrak{p}(d,1)=\frac{4}{12},\\
    &  \mathfrak{p}(d,2)=\frac{4}{12},\\
     &  \mathfrak{p}(d,3)=\frac{2}{12},\\
      &  \mathfrak{p}(d,4)=\frac{1}{12}.\\
\end{split}\]
Thus, as $p\rightarrow \infty$, and $g$ increases with speed $\frac{d}{2}=4$, the $p$-rank grows with expected speed 
\[0\times \frac{1}{12}+1\times \frac{4}{12}+2\times \frac{4}{12} +3\times \frac{2}{12}+4\times \frac{1}{12}=\frac{19}{12}=1.5833\dots\]
\bibliographystyle{abbrv}
\bibliography{references}

\end{document}